\documentclass{birkjour}
\usepackage{subfigure}
\usepackage{verbatim}
 \newtheorem{thm}{Theorem}[section]
 \newtheorem{corollary}[thm]{Corollary}
 \newtheorem{lemma}[thm]{Lemma}
 \newtheorem{Proposition}[thm]{Proposition}
 \theoremstyle{definition}
 
 \theoremstyle{remark}
 
 \newtheorem{example}{Example}
 \numberwithin{equation}{section}

 \newcommand{\R}{\mathbb{R}}
  \newcommand{\E}{\mathcal{E}}
   \renewcommand{\H}{\mathcal{H}}
    \newcommand{\T}{\mathcal{T}}
  \renewcommand{\S}{\mathbb{S}}
  \renewcommand{\P}{\mathbb{P}}
  
   \newcommand{\Z}{\mathbb{Z}}

    \renewcommand{\H}{\mathcal{H}}
    
\renewcommand{\L}{\mathcal{L}}

\begin{document}

%
%

\title[Quadratic Points of Surfaces in Projective $3$-Space]
 {{\LARGE Quadratic Points of Surfaces in Projective $3$-Space }}

\author[M.Craizer]{Marcos Craizer}

\address{%
Departamento de Matem\'{a}tica- PUC-Rio\br
Rio de Janeiro, RJ, Brasil}
\email{craizer@puc-rio.br}

\author[R.Garcia]{Ronaldo A. Garcia}

\address{%
Instituto de Matem\'atica e Estat\'istica - UFG\br
Goi\^ania, GO, Brasil}
\email{ragarcia@ufg.br}

\thanks{The authors thank CNPq for financial support during the preparation of this manuscript. The second author thanks also FAPEG. \newline E-mail of the corresponding author: craizer@puc-rio.br}



\date{October 4, 2017}

\begin{abstract}
Quadratic points of a surface in the projective $3$-space are the points which can be 
exceptionally well approximated by a quadric. They are also singularities of a $3$-web in the elliptic part
and of a line field in the hyperbolic part of the surface. 
We show that generically the index of the $3$-web at a quadratic point is $\pm 1/3$, while the index of the line field is $\pm 1$. 
Moreover, for an elliptic quadratic point whose cubic form is semi-homogeneous, we can use Loewner's conjecture to show that
the index is at most $1$.  

From the above local results we can conclude some global results: A generic compact elliptic surface has at least $6$ quadratic points, a
compact elliptic surfaces with semi-homogeneous cubic forms has at least $2$ quadratic points and the number of quadratic points in 
a hyperbolic disc is odd. By studying the behavior of the cubic form in a neighborhood of the parabolic curve, we also obtain a relation 
between the indices of the quadratic points of a generic surface with non-empty elliptic and hyperbolic regions.
\end{abstract}

\maketitle

\section{Introduction}

Consider a surface $M$ in $3$-space. A quadratic point $p$ is a point
where $M$ can be exceptionally well approximated by a quadric. 
Generically, the quadratic points are isolated and occur at the elliptic or the hyperbolic region of $M$, 
not in the parabolic curve.

Examples of algebraic hyperbolic tori in projective $3$-space with many quadratic points can be found in \cite{Tabach}, where it is conjectured that any hyperbolic tori in $\P^3$ 
should have quadratic points. In \cite{Ronaldo} one can find hyperbolic tori in $\S^3$ with no quadratic points. In \cite{Uribe}, it is proved that any hyperbolic disk bounded by a parabolic curve
should have an odd number of quadratic points. 

The cubic form $C$ defines a binary cubic differential equation (BCDE) outside the parabolic curve.
This BCDE is totally real in the elliptic region and has one root in the hyperbolic region, and is singular exactly 
at the quadratic points. Thus the cubic form defines a line field ${\mathcal H}$ in the hyperbolic region and a $3$-web ${\mathcal E}$ in the elliptic region,
except at quadratic points. Both fields are projectively invariant.

The behavior of a totally real $3$-web close to a singularity was studied in \cite{Ballesteros}. Applying the results of this paper
to the particular case of a cubic form, we verify that generically the $3$-web $\mathcal{E}$ has index $\pm\frac{1}{3}$ at quadratic elliptic points. 
As a corollary we have that at a generic compact convex surface, there exist at least $6$ quadratic points.

For semi-homogeneous cubic forms, we show a relation between the index of the $3$-web and Loewner's conjecture
to conclude that the index of a quadratic point is at most $1$. As a corollary, we obtain that, under the semi-homogeneity hypothesis,  
the number of quadratic points in a compact convex surface is at least $2$, which is a version of Carath\'eodory conjecture. We also 
give an example of a compact rotation surface with exactly $2$ quadratic points.

For quadratic points in the hyperbolic region, we show that generically the line field $\mathcal{H}$ has index $\pm 1$. 
By considering the polar blow-up, we describe the generic phase portrait of $\mathcal{H}$ at these points. 

In a neighborhood of the parabolic curve, there exists a line field ${\mathcal L}$, tangent to the parabolic curve, which coincides 
with the line field $\mathcal{H}$ in the hyperbolic part $H$ and with a line field $\E_1$ of the $3$-web $\E$ in the elliptic part $E$. With the help of this line field, we prove a version of Poincar\'e-Hopf Theorem which says that for a generic compact surface $M$ with elliptic region $E$ and hyperbolic region $H$, 
$$
\sum_{p_i\in E} Ind(p_i,\E) +\sum_{p_i\in H} Ind(p_i,\H)=\chi(M).
$$
where the sum is taken over the the quadratic points $p_i$. 

The paper is organized as follows: In section 2 we show the relation between quadratic points and cubic forms. In section 3
we study the behavior of the $3$-web $\mathcal{E}$ at a simple quadratic point, while in section 4 we calculate the index
of a $3$-web coming from a semi-homogeneous cubic form. In section 5 we study the behavior of the line field $\mathcal{H}$
at a simple quadratic point. In section 6 we calculate the cubic form in a neighborhood of a generic parabolic curve. Finally, 
in section 7, we prove some Poincar\'e-Hopf type theorems.

\section{Quadratic points and cubic forms}

\subsection{Darboux directions and quadratic points}

A general point $p\in M$ admits a $3$-dimensional space quadrics with contact of order $2$ with the surface, i.e., 
with the same tangent plane and the same curvatures in all directions. Among these quadrics, there are three one-parameter families with the property that the contact function is a perfect cube. If $p$ is elliptic, all three families are real, while if $p$ is hyperbolic, one of these families is real and the other two complex.The null directions of the perfect cubes are called Darboux directions (\cite[p.358]{Darboux}, \cite[p. 141-144]{Lane}). 

{\it Quadratic points} $p\in M$ are points that admits an osculating quadric of order $3$. 
Quadratic points are invariant by projective transformations of the $3$-space and admit a 
one-dimensional space of third order osculating quadrics.

\begin{example}
Let $p=(0,0,0)$ and $M$ be given by
\begin{equation}\label{eq:PickDef}
z=\frac{1}{2}(x^2+y^2)+\frac{c}{6}(x^3-3xy^2)+O_4(x,y).
\end{equation}
Any quadric of the form 
$$
z=\frac{1}{2}(x^2+y^2)+\alpha xz+\beta yz+\gamma z^2
$$
has a second order contact with $M$ at $p$. For the quadrics
$$
z=\frac{1}{2}(x^2+y^2)-cxz+\gamma z^2,
$$
the contact function at order $3$ is a multiple of $x^3$. Thus $(0,1)$ is a Darboux direction. The other Darboux directions 
are $(3,\pm\sqrt{3})$. The point $p$ is quadratic if and only if $c=0$.
\end{example}

\begin{example}
Let $p=(0,0,0)$ and $M$ be given by
\begin{equation}\label{eq:PickIndef}
z=xy+a^3x^3+b^3y^3+O_4(x,y).
\end{equation}
Any quadric of the form 
$$
z=xy+\alpha xz+\beta yz+\gamma z^2
$$
has a second order contact with $M$ at $p$. For the quadrics
$$
z=xy-3a^2bxz-3ab^2yz+\gamma z^2,
$$
the contact function at order $3$ is a multiple of $(ax+by)^3$. Thus $(-b,a)$ is the Darboux direction. 
The point $p$ is quadratic if and only if $a=b=0$. 
\end{example}

\subsection{Affine differential geometry and cubic forms}

In the elliptic and hyperbolic parts of $M$, consider the equation
\begin{equation*}
D_XY=\nabla_XY+g(X,Y)\xi
\end{equation*}
where $X,Y$ are vector fields on $M$, $\xi$ is a transversal vector field, $\nabla$ is an affine connection and $g$ a non-degenerate
bilinear form. It turns out that there exists a unique (up to sign) vector field $\xi$ such that $D_X\xi$ is tangent to $M$, for any vector field $X$
on $M$, and the area form induced by $\xi$ coincides with the area form defined by $g$. This $\xi$ is called the {\it affine normal vector field},
the corresponding $g$ and $\nabla$ are called the {\it Blaschke metric} and {\it induced connection}, respectively. At elliptic points, the Blaschke
metric is definite, while at hyperbolic points, the Blaschke metric is indefinite (see \cite{Nomizu}).

The cubic form in the elliptic and hyperbolic parts of $M$ are defined by
$C=\nabla g$.
It turns out that the conformal class of the cubic forms is projectively invariant. At elliptic points, $C$ has three real null directions, while
at hyperbolic points, $C$ has only one real null direction (see \cite{Nomizu}). The following result can be found in 
\cite[p.114]{Tabach1}:

\begin{lemma}
A point $p$ is quadratic if and only if $C(p)=0$.
\end{lemma}
\begin{proof}
Assume first that $p$ is elliptic. Then, up to an affine change of coordinates, the surface $M$ in a neighborhood of $p$
is given by equation \eqref{eq:PickDef}. Then $p$ is quadratic if and only if $c=0$, i.e., if and only if the cubic form at $p$ is zero. 

If $p$ is hyperbolic, up to an affine change of coordinates, the surface $M$ in a neighborhood of $p$ is given by
equation \eqref{eq:PickIndef} The point $p$ is quadratic if and only if $a=b=0$, i.e., if and only if the cubic form at $p$ is zero. 
\end{proof}

\section{Simple quadratic points in the elliptic region}

Assume that
$M$ is the graph of $f:U\subset\R^2\to\R$ given by 
\begin{equation}\label{eq:general4}
f=\frac{1}{2}(x^2+y^2)+\frac{1}{24}\left( a_{40}x^4+4a_{31}x^3y+6a_{22}x^2y^2+4a_{13}xy^3+a_{04}y^4 \right)+O_5(x,y),
\end{equation}
where $O_5(x,y)$ denotes the terms of order $\geq 5$ in $(x,y)$. Then the cubic form is given by
\begin{equation}\label{eq:CubicFormElliptic}
\omega= \omega_1+O_2(x,y),
\end{equation}
where
\begin{equation}
\omega_1= (ax+by)(dx^3-3dxdy^2)+(cx+dy)(dy^3-3dydx^2)
\end{equation}
and 
\begin{equation}
a=a_{40}-3a_{22},\ b=a_{31}-3a_{13},\ c=a_{13}-3a_{31}\ d=a_{04}-3a_{22},
\end{equation}
(see appendix A). 
We say that the singularity $(0,0)$ of $\omega$ is {\it simple} if it is also an isolated singularity of $\omega_1$. It is easy 
to check that this condition is equivalent to
\begin{equation}\label{Hyp:11}
ad-bc\neq 0.
\end{equation}
Observe that for a generic surface $M$, all quadratic points are simple.

\subsection{Index of simple quadratic points}\label{sec:IndexE}

Using complex notation we can write $\omega=(A+iB)dz^3+(A-iB)d\bar{z}^3$, where $A=C_{111}$ and $B=C_{222}$. It is proved in \cite{Ballesteros} that the 
index $I=I(0,0)$ of the singularity is given by
\begin{equation}\label{eq:IndexAB}
I=-\frac{\deg(A+iB)}{3}.
\end{equation}
The complex function $A+iB$ can be approximated by $A_1+iB_1=(ax+by)+i(cx+dy)$. In fact, we have
$$
A+iB=A_1+iB_1+O_2(x,y)
$$
\begin{lemma}\label{lemma:IndexSimple}
For a simple quadratic point, the degree of $A+iB$ is the same degree of $A_1+iB_1$.
\end{lemma}
\begin{proof}
Since the image of $|z|=r$ by $A_1+iB_1$ does not pass through the origin, we can choose $r$ sufficiently small such that
$$
||A+iB-(A_1+iB_1)||\leq \frac{1}{2}||A_1+iB_1||.
$$
This implies that the degrees of $A+iB$ and $A_1+iB_1$ are the same.
\end{proof}

Define the {\it characteristic polynomial}
\begin{equation}\label{eq:DefineP}
P(x,y)=\left( ax+by \right)  (x^3-3xy^2)+\left( cx+dy \right) (y^3-3yx^2),
\end{equation}
and let
\begin{equation}\label{eq:DefineQ}
Q(x,y)=-\left( ax+by \right)  (y^3-3xy^2)+\left( cx+dy \right) (x^3-3xy^2).
\end{equation}
Observe that if the quadratic point is simple, $P$ and $Q$ have no common roots.

Let $P(t)=P(\cos(t),\sin(t))$, $Q(t)=Q(\cos(t),\sin(t))$, $A_1(t)=A_1(\cos(t),\sin(t))$ and $B_1(t)=B_1(\cos(t),\sin(t))$. Then
$$
P(t)+iQ(t)=(A_1(t)+iB_1(t))\exp{(3it)}.
$$
Thus the degree of $P+iQ$ is $\deg(A_1+iB_1)+3$, and so we conclude that 
\begin{equation}
I=1-\frac{\deg(P+iQ)}{3}.
\end{equation}
The polynomials $(P,Q)$ were considered in \cite{Ivanov} in the context of euclidean umbilic points.

\subsection{Polar blow-up of the cubic form}\label{sec:PhaseE}

We can factor the cubic form $\omega$ as $\omega=\lambda_1\lambda_2\lambda_3$, where $\lambda_j$ is a $1$-form 
given by
$$
\lambda_j= \cos(\phi_j)dx -\sin(\phi_j)dy=e^{i\phi_j}dz+e^{-i\phi_j}d\bar{z}.
$$
Thus
$$
\omega_1=\lambda_1\lambda_2\lambda_3=\exp{(i(\phi_1+\phi_2+\phi_3))}dz^3+\exp{(-i(\phi_1+\phi_2+\phi_3))}d\bar{z}^3
$$
$$
+\left[ \exp{(i(\phi_1+\phi_2-\phi_3))}+\exp{(i(\phi_1-\phi_2+\phi_3))}+\exp{(i(-\phi_1+\phi_2+\phi_3))} \right] d\bar{z}dz^2
$$
$$
+\left[ \exp{(i(\phi_1-\phi_2-\phi_3))}+\exp{(i(-\phi_1+\phi_2-\phi_3))}+\exp{(i(-\phi_1-\phi_2+\phi_3))} \right] d\bar{z}^2dz.
$$

We conclude that $\phi_1=\theta-\pi/3$, $\phi_2=\theta$, $\phi_3=\theta+\pi/3$, where
\begin{equation}\label{eq:TanTheta}
\tan(3\theta)=\frac{c\cos t+d\sin t}{a\cos t+b\sin t}.
\end{equation}

Consider the polar projection $\pi:(-\delta,\delta)\times\R\to\R^2$ given by $\pi(r,t)=(r\cos(t),r\sin(t))$. Then 
\begin{equation*}
\pi^*\lambda_j=\cos(\phi_j+t)dr-r\sin(\phi_j+t)dt.
\end{equation*}
Define the $6\pi$-periodic vector fields
\begin{equation}
X_j=r\sin(\phi_j+t)\frac{\partial}{\partial r}+\cos(\phi_j+t)\frac{\partial}{\partial t}, 
\end{equation}
in the kernel of $\pi^*\lambda_j$. Then $X_j$ is singular only at points $(0,t_0)$ such that $\phi_j(t_0)+t_0=k\pi+\frac{\pi}{2}$, $k\in\Z$.
At these points 
$$
DX_j(0,t_0)=(-1)^k
\left[
\begin{array}{cc}
1 & 0 \\
0 & -(1+\phi_j'(t_0))
\end{array}
\right].
$$
The singularity $(0,t_0)$ is hyperbolic for $X_j$ if and only if $1+\phi_j'(t_0)\neq 0$.

Since $\pi^*\omega=\pi^*\lambda_1\cdot\pi^*\lambda_2\cdot\pi^*\lambda_3$, 
the singularities of $\pi^*\omega$ are singularities of $\pi^*\lambda_j$, for some $j=1,2,3$. The following lemma is proved in \cite{Ballesteros}:

\begin{lemma}
The singularities of $\pi^*\omega$ are of the form $(0,t_0)$, where $t_0$ is a root of the characteristic polynomial $P(t)$. 
Moreover, the singularity $(0,t_0)$ is hyperbolic for $X_j$ if and only if $t_0$ is a simple root of $P(t)$.
\end{lemma}

\begin{lemma}
For a singular point $(0,t_0)$ of $\pi^*\lambda_j$, we have that
$$
P'(t_0)=-3Q(t_0)(1+\phi_j'(t_0)).
$$
\end{lemma}

\begin{proof}
Differentiating \eqref{eq:TanTheta} we obtain
$$
\sec^2(3\theta)3\theta'=\frac{ad-bc}{(a\cos t+b\sin t)^2}
$$
and so
$$
3\theta'=\frac{ad-bc}{(a\cos t+b\sin t)^2+(c\cos t+d\sin t)^2}.
$$
On the other hand, 
$$
P'+3Q=(-a\sin t+b\cos t)\cos(3t)+(c\sin t-d\cos t)\sin(3t).
$$
Assuming $P(t)=0$, we have $(a\cos(t)+b\sin(t))\cos(3t)=(c\cos(t)+d\sin(t))\sin(3t)$. Thus
$$
Q=\frac{c\cos(t)+d\sin(t)}{\cos(3t)};\ P'+3Q=\frac{bc-ad}{Q};\ 3\theta'=\frac{ad-bc}{Q^2},
$$
thus proving the lemma.
\end{proof}

As a consequence, the singularity $(0,t_0)$ is a saddle for $X_j$ if and only if $P'(t_0)$ and $Q(t_0)$ have opposite signs. 
So we can determine the phase portrait of the $3$-web in a neighborhood of a simple quadratic point only from the polynomials $P$
and $Q$. 

\subsection{Phase portraits of simple singularities}

Consider the characteristic polynomial $P$ defined by equation \eqref{eq:DefineP}. we may write $P$ as
$$
P=ax^4+(b-3c)x^3y-3(a+d)x^2y^2+(c-3b)xy^3+dy^4.
$$
We can now calculate the discriminant $\Delta$ of $P$ (see \cite{Poston}). The polynomial $P$ has double roots 
if and only if the discriminant $\Delta$ vanishes.
Straightforward calculations show that 
$$
\Delta= S^3-27T^2,
$$
where
$$
S=\frac{1}{4}\left( 10(ad-bc)+3a^2+3b^2+3c^2+3d^2 \right).
$$
and
$$
T=\frac{1}{8}\left( (a+d)(a-d)^2-2bc(a+d)+ac^2+db^2-3dc^2-3ab^2 \right).
$$

\begin{lemma}\label{lemma:ProjectiveChange}
By a projective change of coordinates, we may take $a+d=0$.
\end{lemma}
\begin{proof}
Consider a projective transformation of the form 
$$
(X,Y,Z)=\frac{1}{1+\lambda z}(x,y,z)
$$
with an adequate $\lambda$.
\end{proof}

Since we are only interested in the conformal class of the cubic form, we may take $b-c=1$. Assuming $d=-a$, 
$c=h-\frac{1}{2}$, $b=h+\frac{1}{2}$ we obtain 
$$
\Delta=27a^2h^2-(1-a^2-h^2)^3;\ \ \delta=\frac{1}{4}-a^2-h^2.
$$
On the other hand, $P$ and $Q$ have common roots if and only $\delta=ad-bc=0$.

The curve $\Delta=0$ is an astroid (full line in Figure \ref{Fig:Astroid}) and $\delta=0$ is a circle tangent to the astroid
(dashed line in Figure \ref{Fig:Astroid}). This astroid appears in \cite{Porteous}. Note that the type of singularity is constant 
on the connected components of the complement of $\Delta=0$ (section \ref{sec:PhaseE}), while
the index is constant on the connected
components of the complement of $\delta=0$ (section \ref{sec:IndexE})

\begin{figure}[htb]
\includegraphics[width=.4
\linewidth,clip =false]{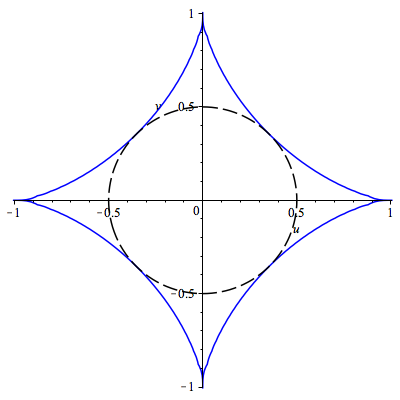}
\caption{The space of parameters $(a,h)$.}
\label{Fig:Astroid}
\end{figure}

\begin{Proposition}
We have that 
\begin{itemize}
\item
If $(a,h)$ is outside the astroid, then $P$ has $2$ roots and the index is $1/3$.
\item
If $(a,h)$ is inside the astroid, the $P$ has $4$ roots:
\begin{itemize}
\item
If $(a,h)$ is outside the dashed circle, then the index is $1/3$.
\item
If $(a,h)$ is inside the dashed circle, then the index is $-1/3$.
\end{itemize}
\end{itemize}
\end{Proposition}

It is proved in \cite{Ballesteros} that the phase portraits of a simple singularity is of type $D_1$, $D_2$ or $D_3$. 
It is clear that $D_1$ corresponds to points outside the astroid, $D_2$ to points inside the astroid but outside the circle,
and $D_3$ to points inside the circle.

\section{Semi-homogeneous elliptic cubic forms}

\subsection{Definition}

Let $\omega$ be a cubic form and $p$ an isolated singularity of $\omega$. Denote by $\omega_n$
the order $n$-jet of $\omega$ at $p$. A cubic form is called  {\it semi-homogeneous} (of degree $n$) if
$\omega_k=0$, $k=1...,n-1$, and $\omega_{n}$ has $p$ as an isolated singularity (see \cite{Ballesteros}). Observe
that for $n=1$ this definition reduces to the definition of a simple singularity.

The complex function $A+iB=C_{111}+iC_{222}$ can be approximated by $A_n+iB_n=C_{111}^{n}+iC_{222}^{n}$, where $C_{ijk}^n$ denotes
the $n$-jet of $C_{ijk}$. More precisely
$$
A+iB=A_{n}+iB_n+O_{n+1}(z,\bar{z}).
$$ 
The semi-homogeneity assumption means that $A_k+iB_k=0$, $k=1,...,n-1,$ and $A_n+iB_n=0$ only for $z=\bar{z}=0$. 

\begin{lemma}\label{lemma:IndexHomogeneous}
For a semi-homogeneous cubic form $\omega$ of degree $n$, the index of $\omega$ 
is given by 
$$
I=-\frac{\deg(A_n+iB_n)}{3}.
$$
\end{lemma}
\begin{proof}
The same argument as in Lemma \ref{lemma:IndexSimple} shows that $\deg(A_n+iB_n)=\deg(A+iB)$. Now the result follows from formula
\eqref{eq:IndexAB}. 
\end{proof}

\subsection{Local model of a surface with semi-homogeneous cubic form}

Assume that $f$ is given by equation \eqref{eq:general4}. 
By a projective change of coordinates (see Lemma \ref{lemma:ProjectiveChange}), 
we may assume $a_{22}=0$ and write
\begin{equation}\label{eq:1}
f(x,y)=\frac{1}{2}(x^2+y^2)+\frac{1}{24}\left( a_{40}x^4+4a_{31}x^3y+4a_{13}xy^3+a_{04}y^4\right)+O_5(x,y).
\end{equation}

\begin{lemma}\label{lemma:2}
Assume that $f$ is given by \eqref{eq:1} and the $1$-jet of the cubic form is zero. Then  
\begin{equation}\label{eq:fn}
f(x,y)=\frac{1}{2}(x^2+y^2)+O_5(x,y).
\end{equation}
\end{lemma}

\begin{proof}
From Appendix \ref{app:CubicOrdern} we have that 
\begin{equation*}
\omega_1=(ax+by)(dx^3-3dxdy^2)+(cx+dy)(dy^3-3dydx^2),
\end{equation*}
where
\begin{equation*}
a=a_{40},\ b=a_{31}-3a_{13},\ c=a_{13}-3a_{31}\ d=a_{04}.
\end{equation*}
Since, by hypothesis, $a=b=c=d=0$, we conclude that $a_{40}=a_{04}=a_{31}=a_{13}=0$.
\end{proof}

\begin{lemma}\label{lemma:h}
Assume that $f$ is given by \eqref{eq:1} and the $(n-1)$-jet of the cubic form is zero. Then necessarily
\begin{equation}\label{eq:fn}
f(x,y)=\frac{1}{2}(x^2+y^2)+h(x,y)+O_{n+4}(x,y),
\end{equation}
where $h=h_{n+3}$ is a homogeneous polynomial of degree $n+3$.
\end{lemma}

\begin{proof}
Induction. The case $n=2$ was proved in Lemma \ref{lemma:2}. For $n\geq 3$, observe that 
the $(n-1)$-jet of the cubic form has $2n$ coefficients which are linear combinations
of the $n+3$ coefficients of $h_{n+2}$, as proved in Appendix \ref{app:CubicOrdern}. We have only to verify that 
this linear map is injective, which is straightforward.
\end{proof}

Lemma \ref{lemma:h} says that if the cubic form is homogeneous of degree $n$, the surface is the graph 
of a function $f$ given by formula \eqref{eq:fn}. Then Appendix \ref{app:CubicOrdern} says that 
$$
\omega_n=A_n(dx^3-3dxdy^2)+B_n(dy^3-3dydx^2),
$$
where
\begin{equation}\label{eq:ThirdDerivatives}
A_n+iB_n=h_{xxx}-3h_{xyy}+i(h_{yyy}-3h_{yxx}).
\end{equation}
The semi-homogeneity hypothesis states that this function has $(0,0)$ as an isolated singularity.

\subsection{Index of a semi-homogeneous cubic form}

\begin{thm} \label{thm:Homogeneous}
Consider an isolated singularity $p$ of the cubic form in the elliptic region of a surface. Assume that the cubic form 
at $p$ is semi-homogeneous. Then the index of the $3$-web is at most $1$.
\end{thm}

\begin{proof}
The index of $p$ coincides with the degree of $A_n+iB_n$ which is given by equation \eqref{eq:ThirdDerivatives}.
Since
$$
\frac{\partial^3 h}{\partial{\bar{z}^3}}=h_{xxx}-3h_{xyy}-i(h_{yyy}-3h_{yxx})
$$
Loewner's conjecture (\cite{Titus}) implies that the degree of $A_n+iB_n$ is $\geq -3$. Then Lemma
\ref{lemma:IndexHomogeneous} implies that the index of the $3$-web is at most $1$. 
\end{proof}

A natural conjecture here is that the index of the cubic form is always $\leq 1$, even without the semi-homogeneity
hypothesis.

\section{Quadratic points in the hyperbolic region}

Assume that
$M$ is the graph of $f:U\subset\R^2\to\R$ given by 
\begin{equation}\label{eq:NormalSimpleHyp}
f(x,y)=xy+\frac{1}{24}\left( ax^4+4bx^3y+6ex^2y^2+4cxy^3+dy^4 \right)+O_5(x,y).
\end{equation}
Then the cubic form is given by
\begin{equation}\label{eq:CubicFormHyp1}
\omega= \omega_1+O_2(x,y),
\end{equation}
where 
\begin{equation}\label{eq:CubicFormHyp}
\omega_1= (ax+by)dx^3+(cx+dy)dy^3,
\end{equation}
(see Appendix \ref{app:CubicOrdern}). We say that the isolated singularity $(0,0)$ of $\omega$ is {\it simple} if it is also an isolated
singularity of $\omega_1$, which is equivalent to
\begin{equation}\label{Hyp:1}
\delta=ad-bc\neq 0.
\end{equation}
We observe that for a generic surface $M$, all quadratic points are simple.

\subsection{Index of simple singularities}\label{sec:IndexH}

Denote by $X$ a vector field in the null direction of $\omega$.
It follows from equation \eqref{eq:CubicFormHyp1} that we can find a vector field $X_1$ in the null direction of $\omega_1$
such that
$$
X-X_1=O_2(x,y).
$$
Then the same argument as in the proof of Lemma \ref{lemma:IndexSimple} shows that 
the index of $X$ is equal to the index of $X_1$, which is equivalent to say that the indices of $\omega$ and $\omega_1$ 
coincide. Observe that the index of $\omega_1$ varies continuously with the parameters $(a,b,c,d)$ in the set $\delta\neq 0$, and
thus it is constant in each connected component of this set. 

Define the {\it characteristic polynomial}
\begin{equation}\label{DefineCharPol}
P(x,y)=(ax+by)x^3+(cx+dy)y^3
\end{equation}
and let 
\begin{equation}\label{eq:DefineQ}
Q(x,y)=xy\left((ax+by)x-(cx+dy)y\right).
\end{equation}

\begin{lemma}\label{lemma:CommonRoots}
For simple singularities satisfying $ad\neq 0$, $P$ and $Q$ have no common roots. 
\end{lemma}
\begin{proof}
If $(x,y)$ is a root of $Q$ with $x\neq 0$ and $y\neq 0$,  $(ax+by)x=(cx+dy)y$, and so
$$
P(x,y)=(ax+by)x(x^2+y^2).
$$
If $(x,y)$ is also a root of $P$,  $ax+by=0$, which implies $cx+dy=0$, thus contradicting hypothesis \eqref{Hyp:1}.
\end{proof}

\begin{lemma}
Denote $P(t)=P(\cos(t),\sin(t))$ and $Q(t)=Q(\cos(t),\sin(t))$. Up to a factor, we can write $\omega_1$ in polar coordinates as
\begin{equation}\label{eq:l3}
\omega_1= P(t)dr^3-3rQ(t)dr^2dt+ O_2(r)drdt^2+O_3(r)dt^3,
\end{equation}
where $O_n(r)$ denotes terms of order $\geq n$ in $r$.
\end{lemma}
\begin{proof}
Straightforward calculations using $x=r\cos(t)$ and $y=r\sin(t)$. 
\end{proof}

\begin{lemma}
Assume $ad\neq 0$ and denote $\bar{X}_1$ the vector field 
$$
\bar{X}_1=(3xQ-yP) \frac{\partial}{\partial x}  +(3yQ+xP)\frac{\partial}{\partial y}.
$$
Then 
$$
||\bar{X}_1-X_1||=O_2(r).
$$
\end{lemma}

\begin{proof}
In polar coordinates, 
$$
\bar{X}_1=3rQ(t)\frac{\partial}{\partial r}+P(t)\frac{\partial}{\partial t}.
$$
Using formula \eqref{eq:l3}, we conclude the proof of the lemma.
\end{proof}

\begin{Proposition}\label{prop:Index1}
The index of the vector fields $X_1$ and $\bar{X}_1$ coincide.
\end{Proposition}
\begin{proof}
Similar to the proof of Lemma \ref{lemma:IndexSimple}.
\end{proof}

\begin{corollary}
Assume $ad\neq 0$. Then the index of a simple singularity is equal to the degree of the map $(x,y)\to(Q(x,y),P(x,y))$ plus $1$.
\end{corollary}

\begin{proof}
By Proposition \ref{prop:Index1}, the index of the simple singularity is given by the degree of the map $(x,y)\to(3xQ-yP,3yQ+xP)$.
But the degree of this map is equal to the degree of the map $(x,y)\to (3Q,P)$ plus $1$.  Finally the degree of $(3Q,P)$ 
is equal to the degree of $(Q,P)$. 
\end{proof}

\subsection{Polar blow-up of the linear form}\label{sec:BlowupH}

We can write
$$
\lambda_1= \cos(\phi)dx -\sin(\phi)dy,
$$
where
\begin{equation}\label{eq:Tan3Theta}
\tan^3(\phi)=-\frac{cx+dy}{ax+by}.
\end{equation}
Denote by $X_1$ a vector field in the kernel of $\lambda_1$. 

Consider the polar projection $\pi:(-\delta,\delta)\times\R\to\R^2$ given by $\pi(r,t)=(r\cos(t),r\sin(t))$. Then straightforward calculations show that 
\begin{equation}\label{eq:l1}
\pi^*\lambda_1=\cos(\phi+t)dr-r\sin(\phi+t)dt.
\end{equation}

\begin{lemma}
Assume $ad\neq 0$. The singularities of $\pi^*\omega_1$ are of the form $(0,t_0)$, 
where $t_0$ is a root of the characteristic polynomial $P$.
\end{lemma}
\begin{proof}
We use Formula \eqref{eq:l3} with $\pi^*\omega_1$ in place of $\omega_1$.  
Thus, $\pi^*\omega_1(r,t)=0$ implies that $P(t)=0$. And since, by Lemma \ref{lemma:CommonRoots}, $P$ and $Q$ have no common roots,
we conclude that $r=0$. Conversely if $P(t)=0$ and $r=0$ is is obvious that $\pi^*(r,t)=0$. 
\end{proof}

The vector field 
$$
\pi^*X_1=r\sin(\phi+t)\frac{\partial}{\partial r}+\cos(\phi+t)\frac{\partial}{\partial t}
$$
is in the kernel of $\pi^*\lambda_1$. $\pi^*X_1$ is singular only at points $(0,t_0)$ such that $\phi(t_0)+t_0=k\pi+\frac{\pi}{2}$, $k\in\Z$.
At these points 
$$
D(\pi^*X_1)(0,t_0)=(-1)^k
\left[
\begin{array}{cc}
1 & 0 \\
0 & -(1+\phi'(t_0))
\end{array}
\right].
$$
Observe that $(0,t_0)$ is hyperbolic for $\pi^*X_1$ if and only if $1+\phi'(t_0)\neq 0$.

\begin{lemma}
Assume $ad\neq 0$. Then, if $P(t_0)=0$, we have that
$$
P'(t_0)=-3Q(t_0)(1+\phi'(t_0)).
$$
\end{lemma}
\begin{proof}
Differentiating equation \eqref{eq:Tan3Theta} we obtain
$$
-3\tan^2(\phi)\sec^2(\phi)\phi'=\frac{ad-bc}{(a\cos t+b\sin t)^2}.
$$
If $P(t)=0$, we have $\phi+t=\frac{\pi}{2}+k\pi$, and so $\sin(\phi)=\pm \cos(t)$, $\cos(\phi)=\pm\sin(t)$. Thus
$$
3\phi'= \frac{bc-ad}{(a\cos t+b\sin t)^2}\frac{\sin^4t}{\cos^2t}.
$$
Moreover, since $P(t)=0$, $Q(t)=\frac{\cos^2t}{\sin t}(a\cos t+b\sin t)$. Thus
$$
3Q\phi'= \frac{bc-ad}{a\cos t+b\sin t}\sin^3t.
$$
On the other hand, 
$$
P'+3Q=(-a\sin t+b\cos t)\cos^3t+(-c\sin t+d\cos t)\sin^3t.
$$
From $P(t)=0$ we obtain 
$$
P'+3Q=\sin^3t \left( (a\sin t-b\cos t) \frac{(c\cos t+d\sin t)}{a\cos t+b\sin t} +(-c\sin t+d\cos t) \right)
$$
Thus $P'+3Q=-3Q\phi'$.
\end{proof}

\begin{Proposition}
Assume $ad\neq 0$. The singularity $(0,t_0)$ of $\pi^*X_1$ is hyperbolic if and only if the root $t_0$ of $P$ is simple.
Moreover, it is a saddle if and only if $P'(t_0)$ and $Q(t_0)$ have opposite signs.
\end{Proposition}

\subsection{Parameter space in the hyperbolic region}

We may write the characteristic polynomial as
$$
P(x,y)=ax^4+bx^3y+cxy^3+dy^4.
$$
This polynomial has multiple roots if and only if its discriminant $\Delta$ vanishes. In this case, we have
that $\Delta=S^3-27T^2$, where
$$
S=ad-\frac{bc}{4}; \ \ T=-\frac{1}{16}(ac^2+db^2),
$$
(see\cite{Poston}). 
By section \ref{sec:BlowupH}, the type of singularity is constant in the connected components of $\Delta=0$ and $ad=0$. On the other hand, 
by section \ref{sec:IndexH}, the index is constant in the connected components of $\delta=0$.

If $bc\neq 0$, by a projective change of coordinates, we may assume $b=\pm 4$, $c=\pm 4$ (see \cite{Tabach}), and we consider two 
different cases, $bc=16$ and $bc=-16$. The third case to consider is $bc=0$. 

\subsubsection{Case $bc=16$}

For $bc=16$, we obtain 
$$
\Delta=(ad-4)^3-27(a+d)^2;\ \ \delta=ad-16.
$$

\begin{Proposition}
We have that 
\begin{itemize}
\item
If $\Delta>0$ (outside the full line in Figure \ref{Fig:ParHyp1}), then $P$ has no roots and the index is $1$.
\item
If $\Delta<0$ (inside the full line in Figure \ref{Fig:ParHyp1}), then $P$ has $2$ roots:
\begin{itemize}
\item
If $\delta>0$ (outside the dashed line in Figure \ref{Fig:ParHyp1}), then the index is $1$.
\item
If $\delta<0$  (inside the dashed line in Figure \ref{Fig:ParHyp1}), then the singularity is a saddle and the index is $-1$.
\end{itemize}
\end{itemize}
\end{Proposition}

\begin{figure}[htb]
\includegraphics[width=.4
\linewidth,clip =false]{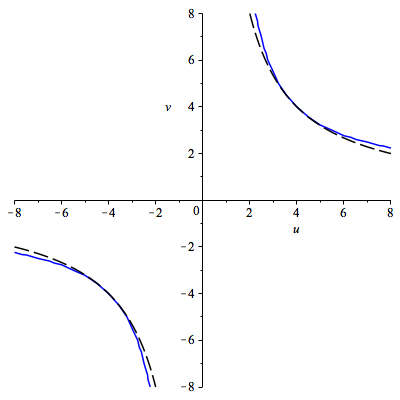}
\caption{The space of parameters $(a,d)$ in case $bc=16$.}
\label{Fig:ParHyp1}
\end{figure}

\subsubsection{Case $bc=-16$}

For $bc=-16$, we obtain
$$
\Delta=(ad+4)^3-27(a+d)^2;\ \ \delta=ad+16.
$$
\begin{Proposition}
The full line curve in Figure \ref{Fig:ParHyp2} is $\Delta=0$. The region $\Delta>0$ is divided 
in a bounded region, that we shall call region $I$, and two unbounded regions, that we shall call region $II$. The region $\Delta<0$
is divided in two unbounded regions, that we shall call region $III$.
\begin{itemize}
\item
If $(a,d)$ belongs to region $I$,  then $P$ has $4$ roots and the index is $1$.
\item
If $(a,d)$ belongs to region $II$, then $P$ has no roots and the index is $1$. 
\item
If $(a,d)$ belongs to region $III$, then $P$ has $2$ roots:
\begin{itemize}
\item
If $\delta>0$ (inside the dashed line in Figure \ref{Fig:ParHyp2}), then the index is $1$.
\item
If $\delta<0$  (outside the dashed line in Figure \ref{Fig:ParHyp2}), then the singularity is a saddle and the index is $-1$.
\end{itemize}
\end{itemize}
\end{Proposition}

\begin{figure}[htb]
\includegraphics[width=.4
\linewidth,clip =false]{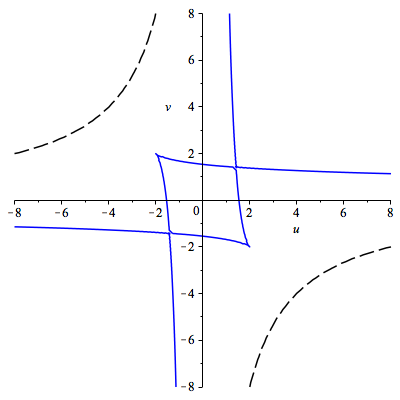}
\caption{The space of parameters $(a,d)$ in case $bc=-16$.}
\label{Fig:ParHyp2}
\end{figure}

\subsubsection{Case $bc=0$}

We shall assume without loss of generality that $b=0$ and $c=4$. In this case,
$$
\Delta=a^2(ad^3-27);\  \delta=ad.
$$

\begin{figure}[htb]
\includegraphics[width=.3
\linewidth,clip =false]{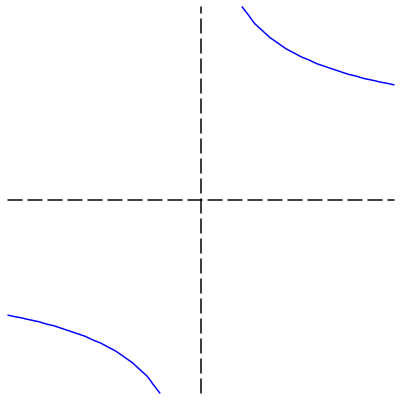}
\caption{The space of parameters $(a,d)$ in case $bc=0$.}
\label{Fig:ParHyp3}
\end{figure}

The full line in Figure \ref{Fig:ParHyp3} is $\Delta=0$, while the dashed lines are $\delta=0$. 

\begin{Proposition}
We have that:
\begin{itemize}
\item If $(a,d)$ is in the second or fourth quadrants, then the index is $-1$.
\item If $(a,d)$ is in the first or third quadrants, then the index is $+1$.
\begin{itemize}
\item If $(a,d)$ is outside the full lines, then $P$ has no roots.
\item If $(a,d)$ is inside the full lines, then $P$ has $2$ roots.
\end{itemize}
\end{itemize}
\end{Proposition}

\subsection{Phase portraits of simple singularities} 

In this section we assume that the singularity is simple, $ad\neq 0$ and the characteristic polynomial $P$ has only simple roots. Under these hypothesis, we shall show that there are only four possible types of phase portraits, and they correspond to index $1$, $P$ with $0$, $2$ or $4$ roots and index $-1$, $P$ with $2$ roots. 

We begin with some general facts: At consecutive roots of $P$, $\frac{dP}{dt}$ changes its sign. To find the sign of $Q$, observe that
if $P=0$ then
$$
(a\cos(t)+b\sin(t))\cos^3(t)=-(c\cos(t)+d\sin(t))\sin^3(t).
$$
At such points 
$$
Q(t)=-\sin^2(t) \left(c+d\tan(t) \right).
$$
So $Q(t)>0$, for $\tan(t)<-\frac{c}{d}$ and $Q(t)<0$, for $\tan(t)>-\frac{c}{d}$.

\paragraph{Case 1: $P$ has no real roots}

In this case the index is $+1$. 
\begin{example}\label{ex:Hyp1}
Consider $a=1$, $b=2$, $c=-4$, $d=12$. 
$$
P=x^4+2x^3y-4xy^3+12y^4=(x^2-2xy+2y^2)(x^2+4xy+6y^2).
$$
The phase portrait of the line field can be seen in Figure \ref{Fig:ExHyp1}.
\begin{figure}[htb]
 \centering
 \includegraphics[width=0.3\linewidth]{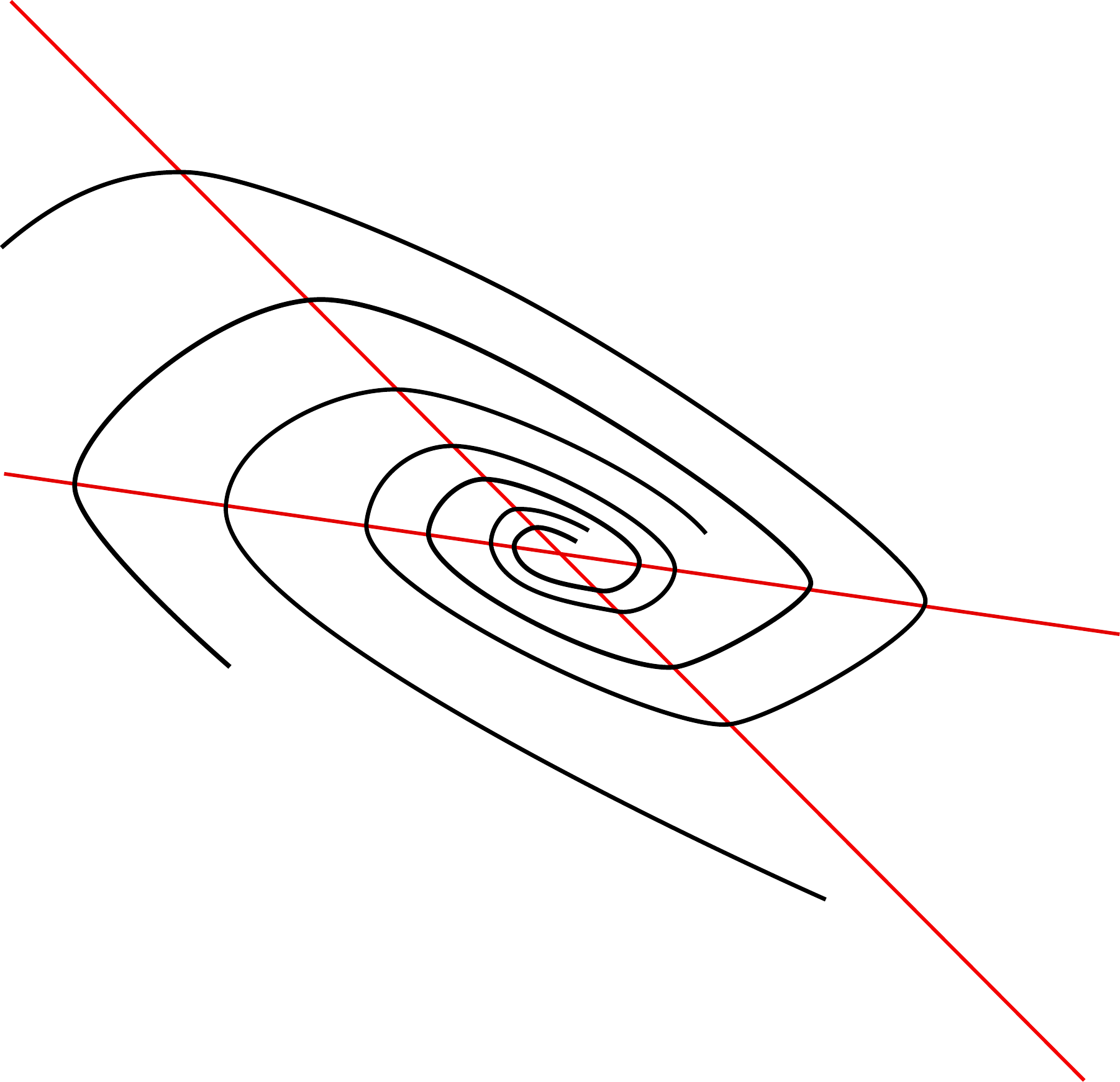}
 \caption{ Line field of Example \ref{ex:Hyp1}. }
\label{Fig:ExHyp1}
\end{figure}
\end{example}

\paragraph{Case 2: $P$ has $4$ real roots} 

From the previous section, we may assume that $bc=-16$ and $(a,d)$ belongs to region $I$ (see Figure \ref{Fig:ParHyp2}). 
It is also clear from our analysis that the type of the hyperbolic singularities of $\pi^*X_1$ are constant in each of the four regions 
$I1=\{(a,d)\in I| a>0,d>0\}$, $I2=\{(a,d)\in I| a<0,d>0\}$, $I3=\{(a,d)\in I| a<0,d<0\}$ and $I4=\{(a,d)\in I| a>0,d<0\}$. Thus to see the type of the 
hyperbolic singularities we must choose a pair $(a,d)$ in each of these four regions and compare $-\frac{c}{d}$ with the roots 
$s_1<s_2<s_3<s_4$ of $P$. 

We have chosen the following pairs: For $(a,b,c,d)=(1,-4,4,1)\in I1$, we obtain $-\frac{c}{d}=-4<-1<s_1$. For $(a,b,c,d)=(-1,-4,4,-1)\in I3$, we obtain $-\frac{c}{d}=4>1>s_4$. For $(a,b,c,d)=(-\frac{8}{5},-4,4,\frac{8}{5})\in I2$, we obtain $-\frac{c}{d}=-\frac{5}{2}<s_1=-2$. For $(a,b,c,d)=(\frac{8}{5},-4,4,-\frac{8}{5})\in I4$, we obtain $-\frac{c}{d}=\frac{5}{2}>s_4=2$. 
We conclude that in any case, the sign of $Q$ at the roots of $P$ keep constant. Thus we have only parabolic sectors.

\begin{example}\label{ex:Hyp4}
Consider
$$
P(t)=\cos^4(t)(\tan(t)-s_1)(\tan(t)-s_2)(\tan(t)-s_3)(\tan(t)-s_4),
$$
where $s_1=-2$, $s_2=-1$, $s_3=1$ and $s_4=-\frac{1}{2}$. Then
$$
a=-1,\ b=-\frac{5}{2},\ c=\frac{5}{2},\ d=1.
$$
The phase portrait of the polar line field can be seen in Figure \ref{Fig:Hyp4}. 

\begin{figure}[htb]
 \centering
 \includegraphics[width=0.3\linewidth]{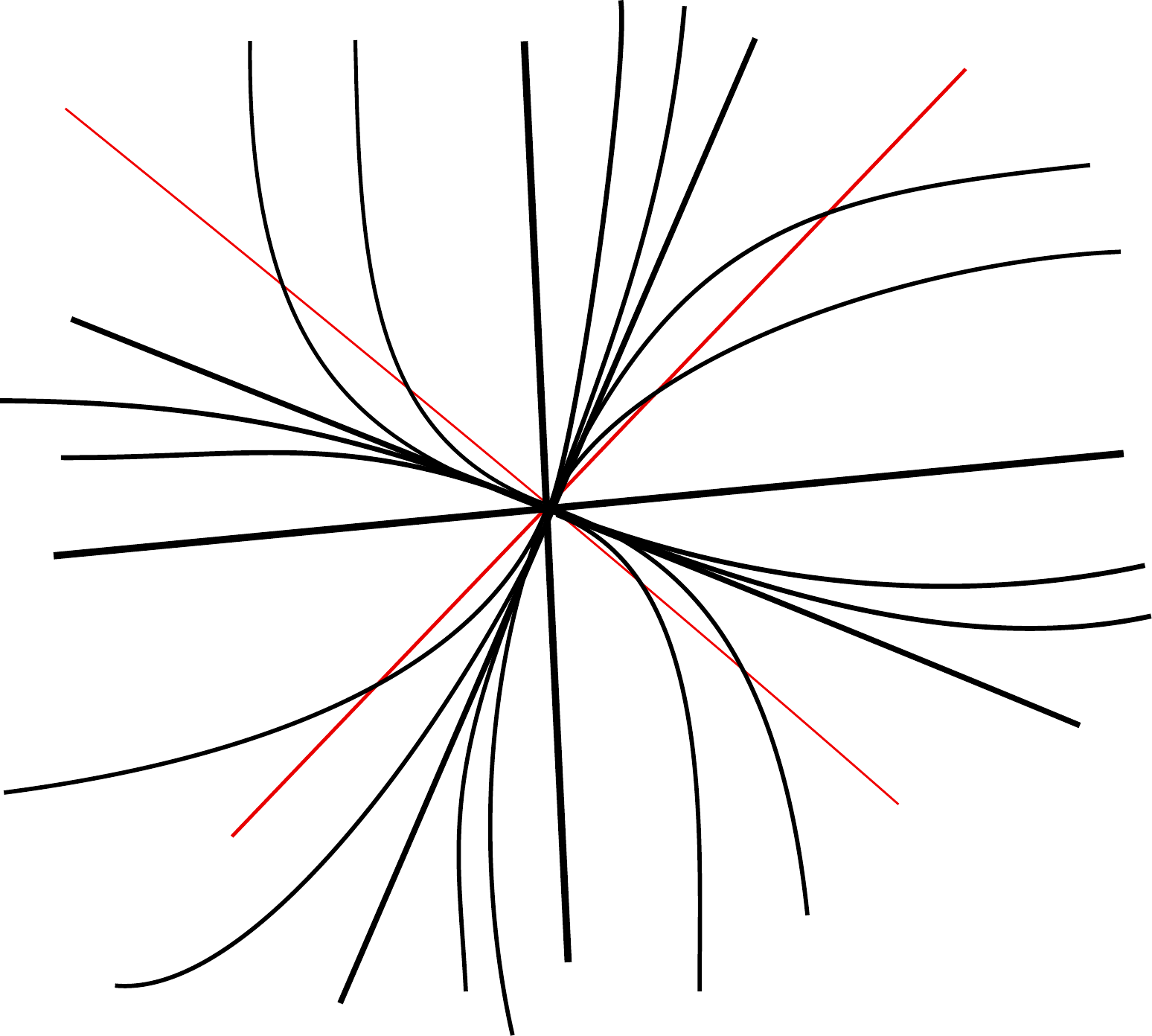}
 \caption{ Line field of Example \ref{ex:Hyp4}. Polar blowing-up resolution with eight hyperbolic singular points, four
 saddles and four nodes.}
\label{Fig:Hyp4}
\end{figure}
\end{example}

\paragraph{Case 3: $P$ has $2$ real roots}  Assume that the roots of $P(t)$ are $t_1<t_2$. We write
$$
P(t)=\cos^2(t)P_2(\cos t,\sin t)(\tan(t)-\tan(t_1))(\tan(t)-\tan(t_2)),
$$
for some quadratic polynomial $P_2(x,y)>0$. Then 
$$
P'(t_1)=P_2(\cos t_1, \sin t_1)(\tan(t_1)-\tan(t_2))<0,
$$
and similarly  $P'(t_2)>0$. 
If the signs of $Q(t_1)$ and $Q(t_2)$ are the same, then there are only parabolic sectors, the singular point
is a node and the index is $+1$. If $Q(t_1)$ and $Q(t_2)$ have different signs, then there are $4$ hyperbolic sectors, the singular point is a saddle and the index is $-1$.

\begin{example}\label{ex:Hyp2}
Consider $P_2(x,y)= (y-x)^2+x^2$, $\tan(t_1)=1$, $\tan(t_2)=-\frac{4}{3}$. 
In other words, $a=-8/3$, $b=10/3$, $c=-5/3$, $d=1$. Then there are only parabolic sectors.
The phase portrait of the line field can be seen in Figure \ref{Fig:Hyp2}.

\begin{figure}[htb]
 \centering
 \includegraphics[width=0.3\linewidth]{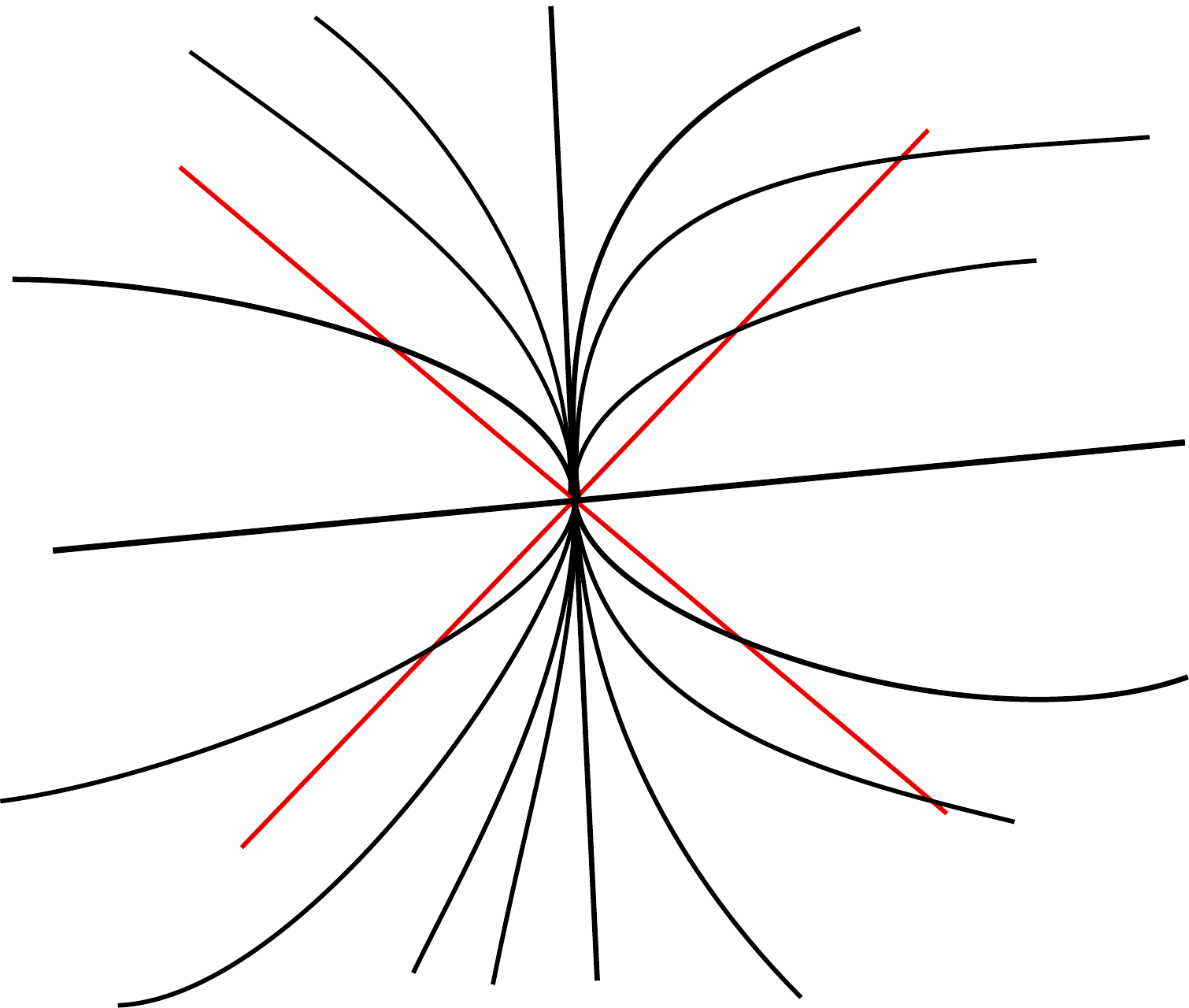}
 \caption{ Line field of Example \ref{ex:Hyp2}. Polar blowing-up resolution with four hyperbolic singular points, two saddles and two nodes. }
\label{Fig:Hyp2}
\end{figure}
\end{example}

\begin{example}\label{ex:Hyp3}
Consider $a=-1$, $d=1$, $b=-1/2$, $c=1/2$. Then there are $2$ singular points of the blow-up, $\tan(t)=\pm 1$.
There are $4$ hyperbolic sectors, thus the index is $-1$. The phase portrait of the line field can be seen in Figure \ref{Fig:Hyp3}.

\begin{figure}[htb]
 \centering
 \includegraphics[width=0.3\linewidth]{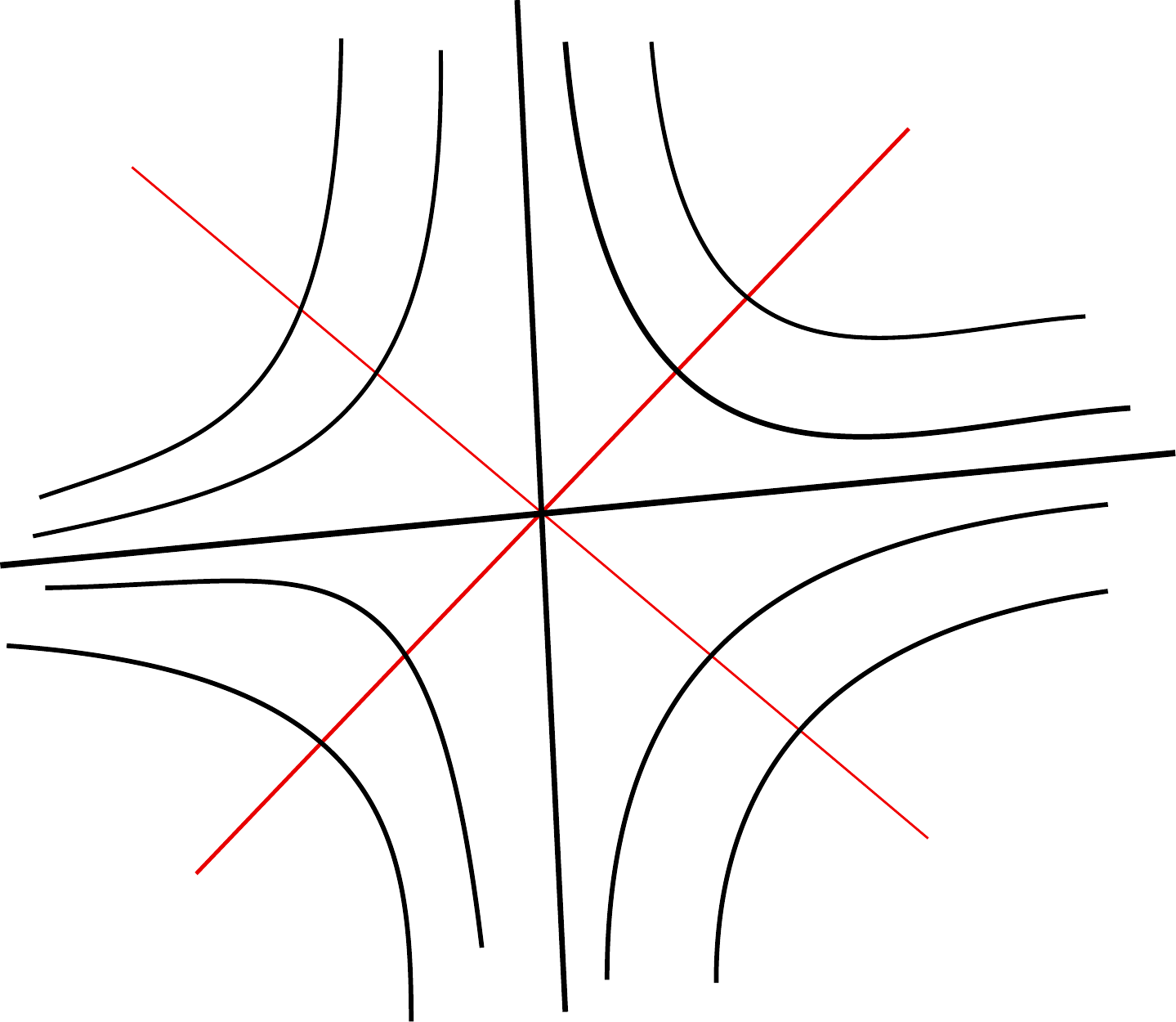}
 \caption{ Line field of Example \ref{ex:Hyp3}. Polar blowing up resolution with four hyperbolic singular points, all saddles.}
\label{Fig:Hyp3}
\end{figure}
\end{example}

\section{Parabolic curve}

Let $C$ be a parabolic curve of a surface $M$. Close to $C$ there are hyperbolic and elliptic regions $H$ and $E$. In this section
we prove that generically, we can extend the line field $\H$ continuously to $C$ by considering the tangent lines $\L$ to $C$. We prove also
that we can select a line field $\E_1$ of the $3$-web $\E$ of $E$ such that $\E_1$ can be continuously extended to $\L$. 

Generically, a parabolic curve has fold points and a finite number of Gauss points, also called godrons (\cite{Izumiya}).

\subsection{The cubic form in a neighborhood of a fold point}

A generic surface $M$ at a fold point is given by
\begin{equation}\label{eq:NormalFold}
z=\frac{1}{2}x^2+\frac{1}{6}y^3.
\end{equation}
(see \cite{Izumiya}).

\begin{Proposition}\label{prop:CubicFold}
Assume that $M$ is given by equation \eqref{eq:NormalFold}. Then the cubic form is conformally equivalent
to $-3dx^2dy+ydy^3$. 
\end{Proposition}
\begin{proof}
See Appendix \ref{app:CubicParabolic}.
\end{proof}

From the above proposition we conclude that, at a fold point, there exists a line field $\L$ 
that coincides with $\H$ in the hyperbolic part, with a line field  $\E_1$ of the $3$-web $\E$ in the elliptic part
at is tangent to the parabolic curve.

\subsection{The cubic form in a neighborhood of a Gauss cusp}

A generic surface in a neighborhood of a Gauss cusp is given by
\begin{equation}\label{eq:NormalGaussCusp}
z=\frac{1}{2}x^2+\frac{1}{2}xy^2+\frac{\lambda}{24}y^4,
\end{equation}
for some $\lambda\neq 0,3$ (see \cite{Platonova} and also \cite{Uribe}).

\begin{Proposition}\label{prop:CubicGaussCusp}
Assume that $M$ is given by equation \eqref{eq:NormalGaussCusp}. Then the cubic form is conformally equivalent
to 
$$
C_{111}=-3, \ C_{112}=-\lambda y, \ C_{122}= 3x-\lambda\frac{y^2}{2},
$$
$$
C_{222}=y\left(x(6+\lambda)+\lambda(\lambda-2) \frac{y^2}{2}\right).
$$
\end{Proposition}
\begin{proof}
See Appendix \ref{app:CubicParabolic}.
\end{proof}

From the above proposition we conclude that, at a Gauss cusp, there exists a line field $\L$ 
that coincides with $\H$ in the hyperbolic part, with a line field  $\E_1$ of the $3$-web $\E$ in the elliptic part
at is tangent to the parabolic curve.

\section{Poincar\'e-Hopf type theorems}

\subsection{Elliptic region}

Consider now an elliptic region $E$ bounded by a parabolic curve $C$. Let $\E$ denote the $3$-web of the cubic form in the region $E$. As we have seen above, in a neighborhood 
of $C$ one of the line field $\E_1$ of the $3$-web $\E$ can be continuously extended to $C$ by considering $\E_1$ tangent to $C$. 

\begin{Proposition}\label{prop:PoincareElliptic}
We have that 
$$
\sum_{p_i\in E} Ind(p_i,\E) =\chi(E).
$$
where the sum is taken over the quadratic points $p_i$.
\end{Proposition}

\begin{proof}
In this proof we follow \cite{Manfredo} and \cite{Ballesteros}. Assume first that $M$ is orientable, for non-orientable surfaces 
consider a double covering. Fix an orientation and a Riemannian metric on $M$. 

Take a triangulation $\T$ such that each triangle $T\in\T$ contains at most one singular point $p_T$ in its interior. Take $X$ 
any vector field in the kernel of $\omega$. Then
$$
\int_{T}Kd\sigma-2\pi Ind(\omega,p_T)=\frac{1}{6}\int_{6\partial T} \left[ \frac{DX}{ds} \right] ds,
$$
where $\left[ \frac{DX}{ds} \right]$ denotes the algebraic value of the covariant derivative and the notation $6\partial{T}$ means $6$ times the boundary of $T$. Summing we obtain
$$
\int_{E}Kd\sigma-2\pi \sum_{p\in E} Ind(\omega,p)=\int_C \left[ \frac{DX}{ds} \right] ds=-\int_C k_gds,
$$
where the last equality comes from the fact that we can choose $X$ tangent to $C$. Now the Gauss-Bonnet theorem
implies that 
$$
\sum_{p\in E} Ind(\omega,p)=\chi(E),
$$
thus proving the proposition.
\end{proof}

\begin{corollary}
A generic compact convex surface $M$ in $\P^3$ has at least $6$ quadratic points.
\end{corollary}

\begin{corollary}\label{cor:Caratheodory}
Let $M$ be a compact convex surface in $\P^3$ and assume that the cubic forms at the quadratic points 
are semi-homogeneous. Then $M$ has at least $2$ quadratic points.
\end{corollary}

Corollary \ref{cor:Caratheodory} is a version of Carath\'edory conjecture for quadratic points, assuming that
the cubic forms are semi-homogeneous. One can conjecture that Corollary \ref{cor:Caratheodory} remains valid
without this assumption.

\subsection{A compact convex surface with two quadratic points}

Consider a smooth rotation surface $S$ with axis $L$. The points of $L\cap S$ are quadratic of index $1$.

We describe below a compact rotation surface $S$ with $2$ points of intersection with the axis $L$ and no other quadratic point.
In rectangular coordinates $(x_1,x_2,y)$, the equation of $S$ is 
\begin{equation*}
x_1^2+x_2^2+y^2(1+\lambda y)^2=(1+\lambda y)^2,
\end{equation*}
for some constant $\lambda>0$ small.

Let $(x(s),y(s))$ be a curve satisfying 
\begin{equation}\label{eq:AffineParameter}
x'y''-y'x''=1
\end{equation}
and consider the rotation surface
\begin{equation}\label{eq:RotationSurface}
\psi(s,\theta)=\left( x(s)\cos\theta, x(s)\sin\theta, y(s) \right).
\end{equation}

\begin{Proposition}\label{prop:CubicRotationSurface}
Assume $M$ is given by equation \eqref{eq:RotationSurface}. Then the quadratic points 
are given by the equation $xy''=x'y'$.
\end{Proposition}

\begin{proof}
See Appendix \ref{app:CubicRotationSurface}.
\end{proof}

\begin{example}
Take 
$$
(x(t),y(t))=\left( \cos(t) (1+\lambda\sin(t)), \sin(t) \right), \ \ -\frac{\pi}{2}\leq t\leq \frac{\pi}{2},
$$
for some constant $\lambda>0$ small (see Figure \ref{Fig:Rotation}).

\begin{figure}[htb]
 \centering
 \includegraphics[width=0.3\linewidth]{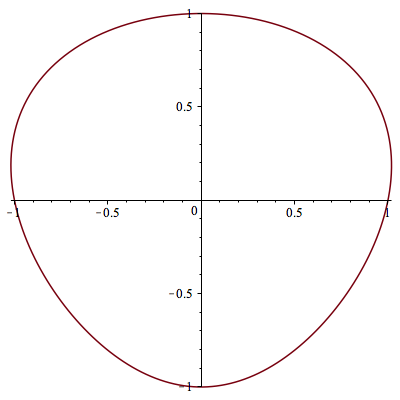}
 \caption{ A section of the rotation surface ($\lambda=0.2$). }
\label{Fig:Rotation}
\end{figure}

Then 
$$
\Delta=x_ty_{tt}-x_{tt}y_t=1+3\lambda\sin t-2\lambda\sin^3t.
$$
Observe that $y'=y_t\frac{dt}{ds}=\cos(t)\Delta^{-1/3}$ and so 
$$
\frac{y'}{x}=\frac{\Delta^{-1/3}}{1+\lambda\sin t}.
$$
Then
$$
\left( \frac{y'}{x} \right)_t= -\frac{\Delta^{-4/3}}{3(1+\lambda\sin t)^2} \left( \Delta_t(1+\lambda\sin t)+3\Delta \lambda\cos t \right)
$$
Since
$$
\Delta_t(1+\lambda\sin t)+3\lambda\cos t\Delta =6\lambda\cos^3t(1+2\lambda\sin t),
$$
we conclude that $\left( \frac{y'}{x} \right)_t<0$, or equivalently $\left( \frac{y'}{x} \right)'<0$. This implies that there no quadratic points other than $(0,0,\pm 1)$. 
\end{example}

\subsection{Hyperbolic region}

Consider an hyperbolic region $H$ bounded by a parabolic curve $C$. Let $\H$ be the line field of the cubic form in the region $H$. As we have seen above, generically $\H$ can be extended continuously to $C$ by considering $\H$ tangent to $C$. Thus the classical Poincar\'e-Hopf Theorem implies the following:
\begin{Proposition}\label{prop:PoincareHyperbolic}
Let $H$ be the hyperbolic region of a generic surface $M$. Then
$$
\sum_{p_i\in H} Ind(p_i,\H) =\chi(H),
$$
where the sum is taken over the quadratic points $p_i$.
\end{Proposition}

The following corollary is proved in \cite{Uribe}:

\begin{corollary}
A generic hyperbolic disc has an odd number of quadratic points.
\end{corollary}

\subsection{A general result}

The following general result is a consequence of Propositions \ref{prop:PoincareElliptic} and \ref{prop:PoincareHyperbolic}:

\begin{thm}
Denote by $E$ and $H$ the elliptic and hyperbolic regions of a generic compact surface $M$ in the projective space $\P^3$. Then 
$$
\sum_{p_i\in E} Ind(p_i,\E) +\sum_{p_i\in H} Ind(p_i,\H)=\chi(M),
$$
where the sum is taken over the quadratic points $p_i$.
\end{thm}

\appendix

\section{Cubic form of graphs of functions}

\subsection{Elliptic and hyperbolic quadratic points}\label{app:CubicOrdern}

Assume
$$
\psi(x,y)=(x,y,f(x,y))
$$
and write $\phi^4=f_{xx}f_{yy}-f_{xy}^2$. Then the Blaschke metric is given by
$$
g_{11}=\frac{f_{xx}}{\phi};\ g_{11}=\frac{f_{xy}}{\phi};\ g_{22}=\frac{f_{yy}}{\phi}.
$$
The affine normal $\xi$ is given by
$$
\xi=\phi(0,0,1)+Z_1\psi_x+Z_2\psi_y,
$$
where
\[
\begin{array}{l}
Z_1=\frac{1}{\phi^4} \left( -f_{yy}\phi_x+f_{xy}\phi_y   \right) \cr
Z_2=\frac{1}{\phi^4} \left( f_{xy}\phi_x-f_{xx}\phi_y   \right)
\end{array}
\]
(see \cite[p.47]{Nomizu}). Writing 
$$
\psi_{xx}=\Gamma_{11}^1\psi_x+\Gamma_{11}^2\psi_y+g_{11}\xi
$$
and using the above formulas we obtain
\[
\begin{array}{l}
\Gamma_{11}^1+g_{11}Z_1=0\cr
\Gamma_{11}^2+g_{11}Z_2=0.
\end{array}
\]
Thus
$$
C_{111}=(g_{11})_x-2\left(\Gamma_{11}^1g_{11}+\Gamma_{11}^2g_{12}\right)=(g_{11})_x+2\left(g_{11}^2Z_1+g_{11}g_{12}Z_2\right).
$$
Straightforward calculations show that
$$
C_{111}=\frac{1}{\phi^5} \left( -3f_{xx}\phi^3\phi_x+f_{xxx}\phi^4 \right)
$$
and so,
\begin{equation}\label{eq:Cubic1}
\phi^5 C_{111}=\frac{1}{4}f_{xxx}f_{xx}f_{yy}-\frac{3}{4}f_{xyy}f_{xx}^2+\frac{3}{2}f_{xxy}f_{xx}f_{xy}-f_{xxx}f_{xy}^2.
\end{equation}
Similarly we obtain
\begin{equation}\label{eq:Cubic2}
\phi^5 C_{112}=-\frac{1}{4}f_{xx}^2f_{yyy}+\frac{3}{4}f_{xx}f_{yy}f_{xxy}-\frac{1}{2}f_{xy}f_{yy}f_{xxx}.
\end{equation}
\begin{equation}\label{eq:Cubic3}
\phi^5 C_{122}=-\frac{1}{4}f_{yy}^2f_{xxx}+\frac{3}{4}f_{xx}f_{yy}f_{xyy}-\frac{1}{2}f_{xy}f_{xx}f_{yyy}.
\end{equation}
\begin{equation}\label{eq:Cubic4}
\phi^5 C_{222}=\frac{1}{4}f_{yyy}f_{xx}f_{yy}-\frac{3}{4}f_{xxy}f_{yy}^2+\frac{3}{2}f_{xyy}f_{yy}f_{xy}-f_{yyy}f_{xy}^2.
\end{equation}

\begin{lemma}
Assume
\begin{equation*}
f(x,y)=\frac{1}{2}(x^2+y^2)+h(x,y),
\end{equation*}
where $h=h_{n+3}$ is a non-zero homogeneous polynomial of degree $n+3$.
Then, up to common factor, the cubic form of $f$ is given by 
\begin{equation*}
\omega=C_{111}(dx^3-3dxdy^2)+C_{222}(dy^3-3dydx^2),
\end{equation*}
where
\begin{equation*}
\left\{
\begin{array}{l}
C_{111}=h_{xxx}-3h_{xyy}+O(n+1)\cr
C_{222}=h_{yyy}-3h_{yxx}+O(n+1).
\end{array}
\right.
\end{equation*}
\end{lemma}

\begin{proof}
Taking the parcels of order at most $n$ in formulas \eqref{eq:Cubic1}, \eqref{eq:Cubic2}, \eqref{eq:Cubic3} and \eqref{eq:Cubic4}, we prove the result.
\end{proof}

\begin{lemma}
Assume $f$ is given by equation \eqref{eq:NormalSimpleHyp}. Then the cubic form is given by equation \eqref{eq:CubicFormHyp}. 
\end{lemma}

\begin{proof}
Take the parcels of order at most $1$ in formulas \eqref{eq:Cubic1}, \eqref{eq:Cubic2}, \eqref{eq:Cubic3} and \eqref{eq:Cubic4}.
\end{proof}

\subsection{Cubic forms at parabolic points}\label{app:CubicParabolic}

\paragraph{Proof of Proposition \ref{prop:CubicFold}}

Assume $M$ is given by \eqref{eq:NormalFold}.
Then 
$$
X=(1,0,f_x)=(1,0,x),\ \ Y=(0,1,f_y)=(0,1,\frac{y^2}{2}),
$$
$L=1,N=y, M=0$, $LN-M^2=y$, and so $(x,y)$ is parabolic if $y=0$, elliptic if $y>0$ and hyperbolic if $y<0$.
Let $\phi=|y|^{1/4}$. Then 
$$
g(X,X)=\frac{1}{\phi},\ g(X,Y)=0,\ g(Y,Y)=\frac{y}{\phi}.
$$
The affine normal is $\xi=\phi\xi_0+Z_1X+Z_2Y$, $\xi_0=(0,0,1)$, where $Z_1=-\phi_x=0$, $Z_2=-\frac{\phi_y}{y}=\mp \frac{1}{4y\phi^3}$.
Moreover
$$
\nabla_XX=\pm \frac{1}{4y\phi^4}Y, \ \ \nabla_XY=0,\ \ \nabla_YY=\pm\frac{1}{4\phi^4}Y.
$$
Thus $C(X,X,X)=C(X,Y,Y)=0$, $C(X,X,Y)=\mp\frac{1}{4\phi^5}$
and
$$
C(Y,Y,Y)=\mp\frac{y}{4\phi^5}+\frac{1}{\phi}\mp 2\frac{y}{4\phi^5}=\pm\frac{1}{4\phi^5}y,
$$
which proves Proposition \ref{prop:CubicFold}.

\paragraph{Proof of Proposition \ref{prop:CubicGaussCusp}}

Assume $M$ is given by equation \eqref{eq:NormalGaussCusp}.
Then 
$$
X=(1,0,f_x)=(1,0,x+\frac{1}{2}y^2),\ \ Y=(0,1,f_y)=(0,1,xy+\lambda\frac{ y^3}{6}),
$$
$$
f_{xx}=(0,0,1),\ \ f_{xy}=(0,0,y),\ \ f_{yy}=(0,0,x+\lambda\frac{y^2}{2}),\ \ \phi^4=\left| x+(\lambda-2)\frac{y^2}{2}\right|,
$$
and the Blaschke metric is given by
$$
g(X,X)=\frac{1}{\phi},\ g(X,Y)=\frac{y}{\phi},\ g(Y,Y)=\frac{x+\lambda\frac{y^2}{2}}{\phi}.
$$
We can write $\xi=\phi\xi_0+Z_1X+Z_2Y, \ \  \xi_0=(0,0,1)$,
where
$$
Z_1+yZ_2=-\phi_x,\ \ yZ_1+(x+\lambda\frac{y^2}{2})Z_2=-\phi_y.
$$

\paragraph{Elliptic points}

Assume $x+( \lambda-2 )\frac{y^2}{2}>0$. Then
$$
\phi_x=\frac{1}{4\phi^3}, \ \ \phi_y=\frac{1}{4\phi^3}(\lambda-2)y,
$$
and so 
$$
(Z_1,Z_2)=\frac{1}{4\phi^7}\left( -x+(\lambda-4)\frac{y^2}{2}, (3-\lambda)y \right).
$$
Then
$$
\nabla_XX=-\frac{f_{xx}}{\phi}(Z_1X+Z_2Y)=\frac{1}{4\phi^8}\left[ \left( x-(\lambda-4)\frac{y^2}{2}\right) X+(\lambda-3)yY\right],
$$
$$
\nabla_XY=y\nabla_XX,\ \ \nabla_YY=\left( x+\lambda\frac{y^2}{2} \right)\nabla_XX.
$$
The cubic form is thus 
$$
C_{111}=-\frac{1}{4\phi^5}-\frac{2}{4\phi^9}\left[ \left( x-(\lambda-4)\frac{y^2}{2}\right)+(\lambda-3)y^2 \right] =-\frac{3}{4\phi^5}.
$$
$$
C_{112}=\frac{(2-\lambda)y}{4\phi^5}-\frac{2y}{4\phi^9}\left[ \left( x-(\lambda-4)\frac{y^2}{2}\right)+(\lambda-3)y^2 \right] =-\frac{\lambda y}{4\phi^5}.
$$
$$
C_{122}=\frac{1}{\phi}-\frac{1}{4\phi^5}(x+\lambda \frac{y^2}{2})-2yg(\nabla_XX,Y)=\frac{1}{4\phi^5} (3x-\lambda\frac{y^2}{2}).
$$
$$
C_{222}=\frac{\lambda y}{\phi}-\frac{(\lambda-2)y}{4\phi^5}(x+\lambda \frac{y^2}{2})-2(x+\lambda\frac{y^2}{2})g(\nabla_XX,Y)
$$
$$
C_{222}=\frac{y}{4\phi^5}(x(6+\lambda)+\lambda(\lambda-2) \frac{y^2}{2}).
$$

\paragraph{Hyperbolic points}

Assume $x+( \lambda-2 )\frac{y^2}{2}<0$. Then $\phi^4=-x+(2-\lambda)\frac{y^2}{2}$, 
$$
\phi_x=-\frac{1}{4\phi^3}, \ \ \phi_y=\frac{1}{4\phi^3}(2-\lambda)y.
$$
The formulas for $Z_1$, $Z_2$, $\nabla_XX$, $\nabla_XY$, $\nabla_YY$ remain the same. The formulas
for $C_{ijk}$ also remain the same.

\section{Cubic form of a rotation surface}\label{app:CubicRotationSurface}

In this appendix we prove Proposition \ref{prop:CubicRotationSurface}. 
Assume $M$ is a rotation surface given by equation \eqref{eq:RotationSurface}.
Then
\[
\left\{
\begin{array}{l}
\psi_s= \left( x'\cos\theta, x'\sin\theta, y'   \right)  \cr
\psi_{\theta}=\left(  -x\sin\theta, x\cos\theta, 0 \right)
\end{array}
\right.
\]
Thus
$$
\psi_s\times \psi_{\theta}=x \left(  -y'\cos\theta, -y'\sin\theta, x' \right).
$$
Moreover,
\[
\left\{
\begin{array}{l}
\psi_{ss}= \left( x''\cos\theta, x''\sin\theta, y''   \right)  \cr
\psi_{s\theta}=\left(  -x'\sin\theta, x'\cos\theta, 0 \right)\cr
\psi_{\theta\theta}=-\left(  x\cos\theta, x\sin\theta, 0 \right)
\end{array}
\right.
\]
Thus
$$
L=x;\ \ M=0;\ \ N=x^2y';\ \ \phi^4=x^3y'.
$$
Then
$$
\nu=\frac{x}{\phi} \left(  -y'\cos\theta, -y'\sin\theta, x' \right), \ \ g_{11}=\frac{x}{\phi};\ \ g_{12}=0;\ \ g_{22}=\frac{x^2y'}{\phi}.
$$

\begin{lemma}
The affine normal vector field $\xi$ can be written as
$$
\xi=\left( a\cos\theta, a\sin\theta, b \right),
$$
for certain $a=a(s)$, $b=b(s)$. 
\end{lemma}

\begin{proof}
Observe that $\nu_{\theta}\cdot\xi=0$. The conditions $\nu\cdot\xi=1$ and $\nu_s\cdot\xi=0$ are given by 
\[
\left\{
\begin{array}{l}
-ay'+bx'=\frac{\phi}{x} \cr
 -ay''+bx''=-\left(\frac{\phi}{x}\right)^2 \left(\frac{x}{\phi} \right)'
\end{array}
\right.
\]
This system certainly has a solution $(a,b)$. 
\end{proof}

We can write 
\[
\left\{
\begin{array}{l}
\psi_{ss}=  \alpha\psi_s+  h_{11}\xi \cr
\psi_{s\theta}= \frac{x'}{x}\psi_{\theta} \cr
\psi_{\theta\theta}= \beta\psi_s+ h_{22}\xi
\end{array}
\right.
\]
Thus
$$
\nabla_{\psi_s}\psi_s=\alpha \psi_s;\ \ \nabla_{\psi_s}\psi_{\theta}=\nabla_{\psi_{\theta}}\psi_s=\frac{x'}{x}\psi_{\theta};\ \ \nabla_{\psi_{\theta}}\psi_{\theta}=\beta \psi_s.
$$
Then
$$
C(\psi_{\theta},\psi_{\theta},\psi_{\theta})=-2g(\nabla_{\psi_{\theta}}\psi_{\theta},\psi_{\theta})=0; \ \ C(\psi_s,\psi_s,\psi_{\theta})=-2g(\nabla_{\psi_{\theta}}\psi_s,\psi_s)=0,
$$
$$
C(\psi_s,\psi_{\theta},\psi_{\theta})=(g_{22})_s-2g(\nabla_{\psi_s}\psi_{\theta},\psi_{\theta})=\left(   \frac{x^2y'}{\phi} \right)' -2\frac{x'}{x}\frac{x^2y'}{\phi}.
$$
By the apolarity condition, a point is quadratic if and only if $C(\psi_s,\psi_{\theta},\psi_{\theta})=0$. This condition is equivalent to
$$
\frac{1}{\phi^2}\left( (2xx'y'+x^2y'')\phi- \phi' x^2y' \right)=2\frac{xx'y'}{\phi}.
$$
After some simplifications we obtain that this condition is equivalent to \newline $xy''-x'y'=0$, thus proving Proposition \ref{prop:CubicRotationSurface}.

\end{document}